\documentclass[11pt]{amsart}
\usepackage{amsfonts}
\usepackage{amssymb}
\usepackage{color}
\usepackage{graphicx}
\usepackage{amsmath,amssymb,amsthm} \usepackage[mathscr]{eucal}
\usepackage{amscd}

\newtheorem{theorem}[]{Theorem}
\newtheorem{lemma}[theorem]{Lemma}
\newtheorem{corollary}[theorem]{Corollary}

\theoremstyle{definition}

\theoremstyle{remark}

\numberwithin{equation}{section}
\numberwithin{theorem}{section}

\begin{document}

\title{Heegaard Diagrams Corresponding to Turaev Surfaces}

\author[C. Armond]{Cody Armond}
\address{Department of Mathematics, University of Iowa,
Iowa City, IA 52242-1419, USA}
\email{cody-armond@uiowa.edu}

\author[N. Druivenga]{Nathan Druivenga}
\address{Department of Mathematics, University of Iowa,
Iowa City, IA 52242-1419, USA}
\email{nathan-druivenga@uiowa.edu}

\author[T. Kindred]{Thomas Kindred}
\address{Department of Mathematics, University of Iowa,
Iowa City, IA 52242-1419, USA}
\email{thomas-kindred@uiowa.edu}

\subjclass{}
\date{}

\maketitle

\begin{abstract}
We describe a correspondence between Turaev surfaces of link diagrams on $S^2\subset S^3$ and special Heegaard diagrams for $S^3$ adapted to links. 
\end{abstract}

\section{Introduction}

To construct the Turaev surface $\Sigma$ of a link diagram $D$ on $S^2\subset S^3$, one pushes the all-A and all-B states of $D$ to opposite sides of $S^2$, connects these two states with a certain cobordism, and caps the state circles with disks.  
Turaev's original construction~\cite{tur} streamlined Murasugi's proof~\cite{mur}, based on Kauffman's work ~\cite{k} on the Jones polynomial~\cite{j}, of Tait's longstanding conjecture on the crossing numbers of alternating links~\cite{tait}.  See also~\cite{this}.  
More recently, Turaev surfaces have
provided geometric means for interpreting Khovanov and knot Floer homologies, as in~\cite{ck, cks07, cks11, dl11, dl13, low, weh}. 

Dasbach, Futer, Kalfagianni, Lin, and Stoltzfus showed that the Turaev surface of any connected link diagram $D$ on $S^2\subset S^3$ is a splitting surface for $S^3$ on which $D$ forms an alternating link diagram~\cite{dfkls}. 
When equipped with the type of crossing ball structure developed by Menasco~\cite{men}, the projection sphere provides natural attaching circles for the two handlebodies of this splitting, completing a
Heegaard diagram $(\Sigma,\alpha,\beta)$ for $S^3$.  By characterizing the interplay between this Heegaard diagram and the original link diagram $D$, we obtain a correspondence between Turaev surfaces and particular Heegaard diagrams adapted to links. Figure~\ref{Fi:quotTur} shows a typical example of such a diagram $(\Sigma,\alpha,\beta,D)$.

First, \textsection2 defines Heegaard splittings and diagrams, link diagrams, crossing ball structures, and Turaev surfaces. Next, \textsection3 
constructs and describes the special, link-adapted Heegaard diagrams $(\Sigma,\alpha,\beta,D)$. Finally, \textsection4 establishes the following correspondences:
 
\textbf{Theorem \ref{Th:Main1}.}
{\it There is a 1-to-1 correspondence between Turaev surfaces of connected link diagrams on $S^2\subset S^3$ and diagrams $(\Sigma,\alpha,\beta,D)$ with the following properties:

$\bullet$ $(\Sigma,\alpha,\beta)$ is a Heegaard diagram for $S^3$, with $\alpha\pitchfork \beta$.

$\bullet$ $D$ is an alternating link diagram on $\Sigma$ which cuts $\Sigma$ into disks, with $D\pitchfork \alpha$ and $D\pitchfork \beta$. 

$\bullet$  $D\cap\alpha=D\cap\beta=\alpha\cap \beta$, none of these points being crossings of $D$. 

$\bullet$ There is a checkerboard partition $\Sigma\setminus(\alpha\cup\beta)=\Sigma_\varnothing\cup \Sigma_K$, in which 
$\Sigma_\varnothing$ consists of disks disjoint from $D$, in which $D$ cuts $\Sigma_K$ into disks each of whose boundary contains at least one crossing point and at most two points of $\alpha\cap \beta$, and in which $ 2  g(\Sigma)+|\Sigma_\varnothing|=|\alpha|+|\beta|$.}

\textbf{Theorem \ref{Th:Gen}:}
{\it There is a 1-to-1 correspondence between generalized Turaev surfaces, constructed from dual pairs of states of connected link diagrams on $S^2\subset S^3$, and diagrams $(\Sigma,\alpha,\beta,D)$ with the properties in Theorem \ref{Th:Main1}, except that $D$ need not alternate on $\Sigma$.}

\begin{figure}
\begin{center}
\scalebox{.2}{\includegraphics{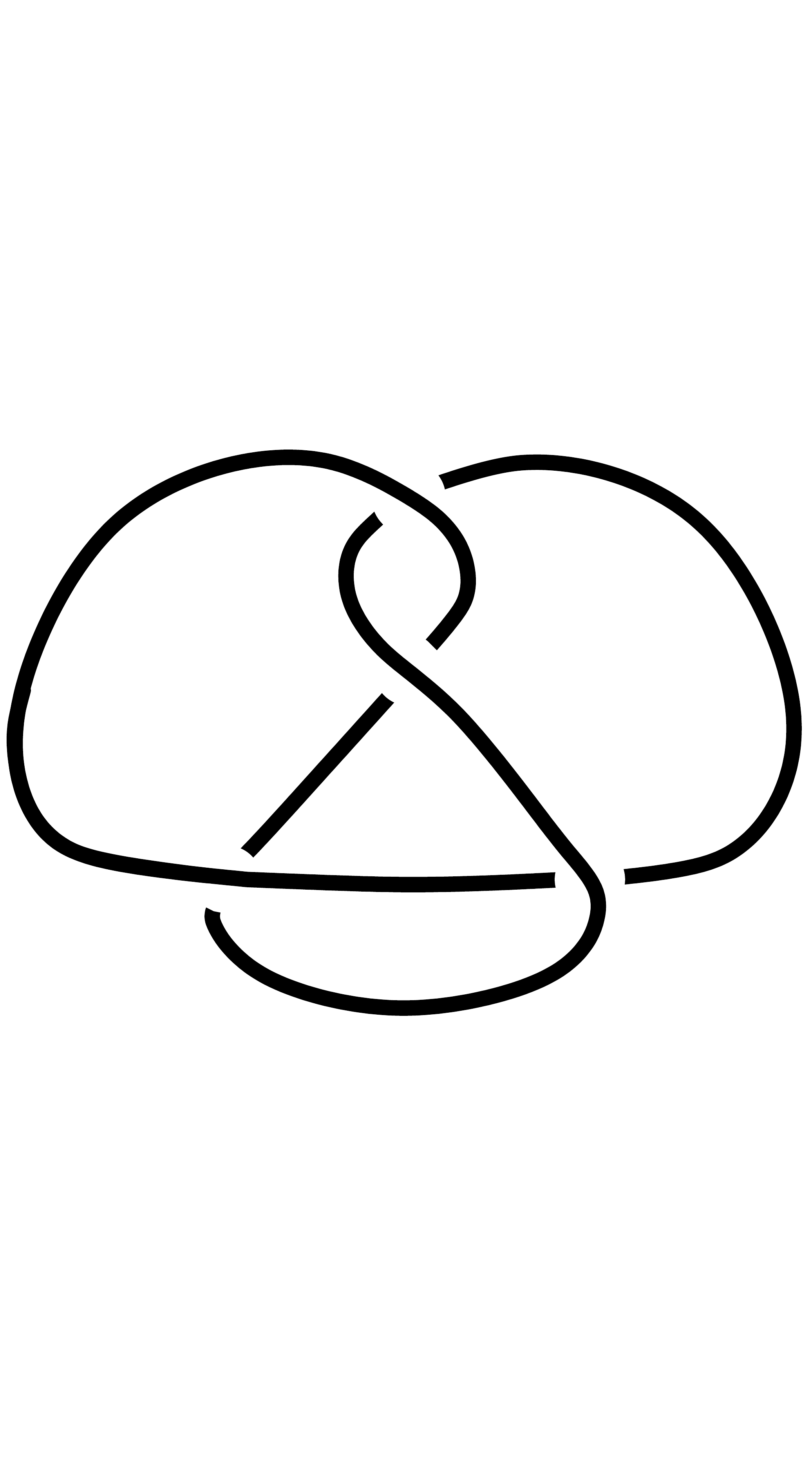}}
$\qquad$
\scalebox{.45}{\includegraphics{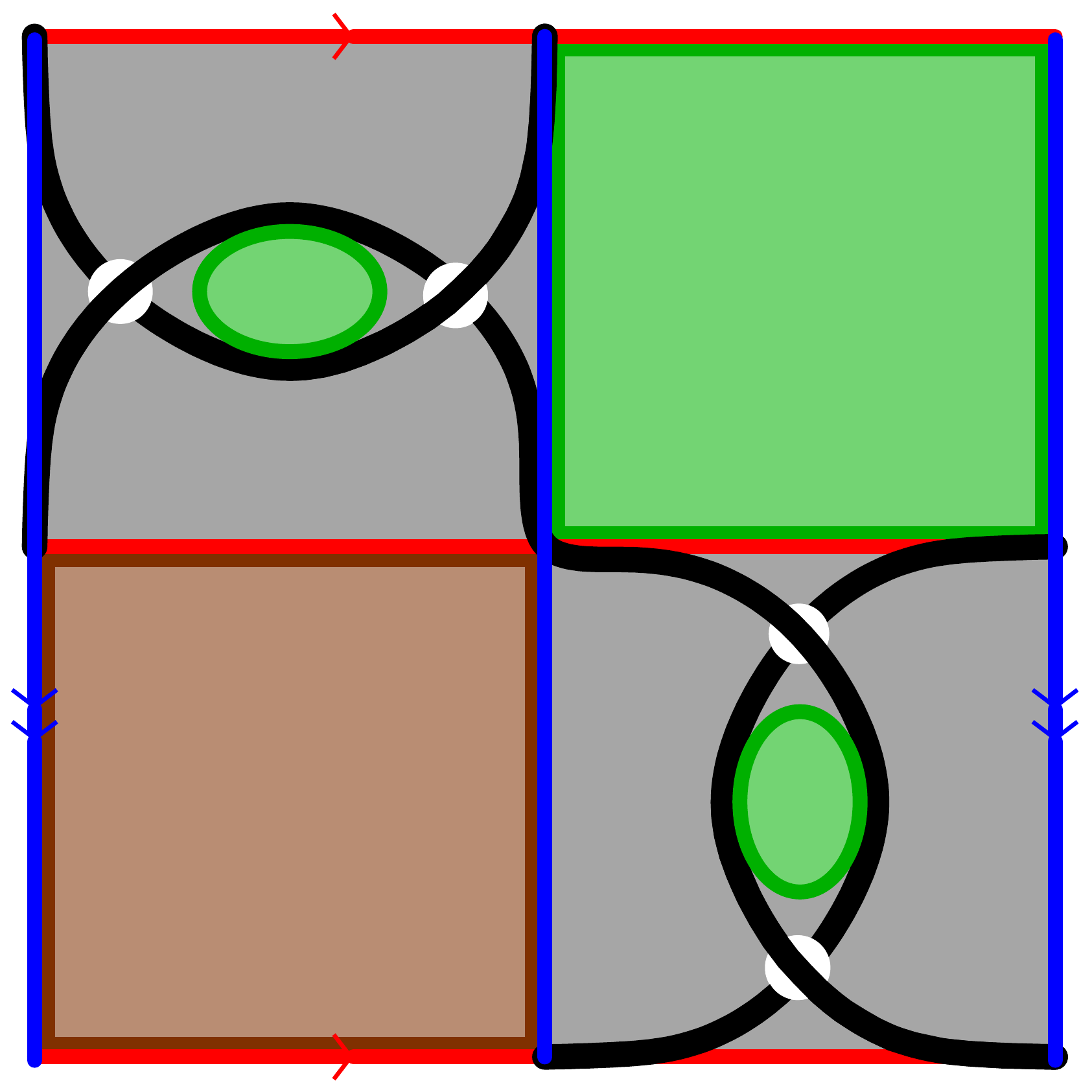}  }  
\caption{A link diagram 
on $S^2$, and the link-adapted Heegaard diagram $(\Sigma,\alpha,\beta,D)$ corresponding 
to its Turaev surface, the torus 
in Figure~\ref{Fi:Tur}. 
As in all figures, the link is black; the crossing balls are white; the attaching circles comprising $\alpha$ and $\beta$ are red and blue, respectively; and the circles and disks from the 
all-A state are green, while those from the all-B state are brown.  
}\label{Fi:quotTur}
\end{center}
\end{figure}

\textbf{Acknowledgements:} We would like to thank Charlie Frohman, Maggy Tomova, Ryan Blair, Oliver Dasbach, Adam Lowrance, Neal Stoltzfus, and Effie Kalfagianni for helpful conversations.

\section{Background}

\subsection{Heegaard splittings and diagrams}\label{Se:Heeg}
A {\it Heegaard splitting} of an orientable $3$-manifold $M$ is a decomposition of $M$ 
into two handlebodies $H_\alpha$ and $H_\beta$ with common boundary.  The surface  $\partial H_\alpha = \partial H_\beta = \Sigma$ is called a {\it splitting surface} for $M$. 
In this paper, we address only the case in which $M=S^3$. 

One can describe a handlebody $H$ by identifying on its boundary $\partial H = \Sigma$ a collection of
 disjoint, 
simple closed curves $\alpha_1, \ldots, \alpha_k$, such that each $\alpha_i$ bounds a disk $\hat{\alpha}_i$ in $H$, and such that these disks together cut $H$ into a disjoint union of balls. The $\alpha_i$ are called {\it attaching circles} for $H$. Some conventions require that the $\hat{\alpha}_i$ together cut $H$ into a single ball, hence $k=g(\Sigma)$; 
though not requiring this, our definition does imply that $k\geq g(\Sigma)$. 

A {\it Heegaard diagram} $(\Sigma,\alpha,\beta)$ combines these ideas to blueprint a 3-manifold.  The diagram consists of a splitting surface $\Sigma = \partial H_\alpha = \partial H_\beta$, together with a union ${\alpha}=\bigcup\alpha_i$ of attaching circles for $H_\alpha$ and a union ${\beta}=\bigcup\beta_i$ of attaching circles for $H_\beta$. 
If $(\Sigma,\alpha,\beta)$ is a Heegaard diagram for $S^3$, then the circles of $\alpha$ and $\beta$ together generate 
$H_1(\Sigma)$. The Appendix provides an easy proof of this fact, using Seifert surfaces.


\begin{figure}
\begin{center}
\scalebox{.135}{\includegraphics{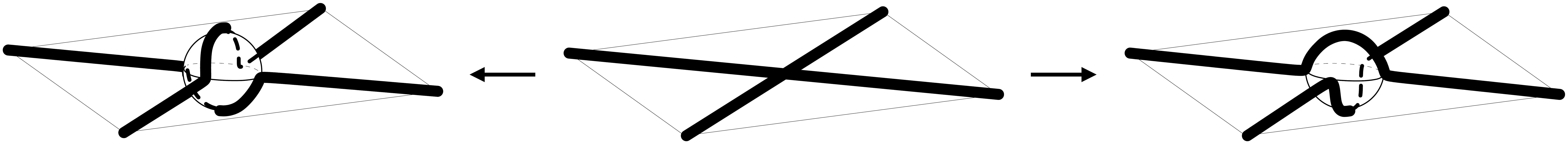}}
\caption{Each crossing in a link diagram is labeled in one of two ways.  The label tells one how to adjust the link after inserting a crossing ball.
}\label{Fi:bubble}
\end{center}
\end{figure}

\newpage

\subsection{Link diagrams and crossing balls}\label{Se:link}
A 
{\it link 
diagram} $D$ 
on a closed surface $F\subset S^3$ is the
 image, in general position, of an immersion of one or more
circles in $F$; each arc at any crossing point is labeled with a direction normal to $F$ near that point, so that under- and over-crossings have been identified. 
By inserting small,
mutually disjoint crossing balls $C=\bigcup C_i$ centered at the crossing points of $D$
and pushing the two intersecting arcs of each $D\cap C_i$ off $F$ to the appropriate hemisphere of $\partial C_i\setminus F$, as in Figure~\ref{Fi:bubble}, one obtains a configuration of a link $K\subset (F\setminus C)\cup \partial C\subset S^3$.  Call this a {\it crossing ball configuration} of the link $K$ corresponding to the link diagram $D$. 

Conversely, given mutually disjoint crossing balls $C=\bigcup C_i$ centered at points on a closed surface $F\subset S^3$, and a link $K\subset (F\setminus C)\cup \partial C$ in which each crossing ball appears as in Figure~\ref{Fi:bubble}, one may obtain a corresponding link diagram as follows.
Consider a regular neighborhood 
of $F$ 
that contains $C$ and is parameterized by an orientation-preserving 
homeomorphism with $F\times [-1,1]$ which identifies $F$ with $F\times \{0\}$. If $\pi:F\times[-1,1]\to F$ denotes the natural projection, the link diagram corresponding to the crossing ball configuration $K\subset (F\setminus C)\cup \partial C\subset S^3$ is the projected image $\pi(K)\subset F$ with appropriate crossing labels.  

In such a 
crossing ball configuration, each arc of $K\cap \partial C$ lies either in $F\times[-1,0]$ or in $F\times[0,1]$.  The former arcs are called {\it under-passes}, and the latter are called {\it over-passes}. A link diagram $D$ is said to be {\it alternating} if each arc of $K\setminus C$ in a corresponding crossing ball configuration joins an under-pass with an over-pass. 
A link $K\subset S^3$ is 
{alternating} if it has an alternating diagram on $S^2$. 

In particular, any Heegaard diagram $(\Sigma,\alpha,\beta)$ for $S^3$ provides an embedding of the closed surface $\Sigma$ in $S^3$. One may therefore superimpose a link diagram $D$ on the Heegaard diagram to obtain a new type of diagram $(\Sigma,\alpha,\beta,D)$.  This new diagram describes a Heegaard splitting of $S^3$ in which the splitting surface contains a link diagram.

\begin{figure}
\begin{center}
\scalebox{.25}{\includegraphics{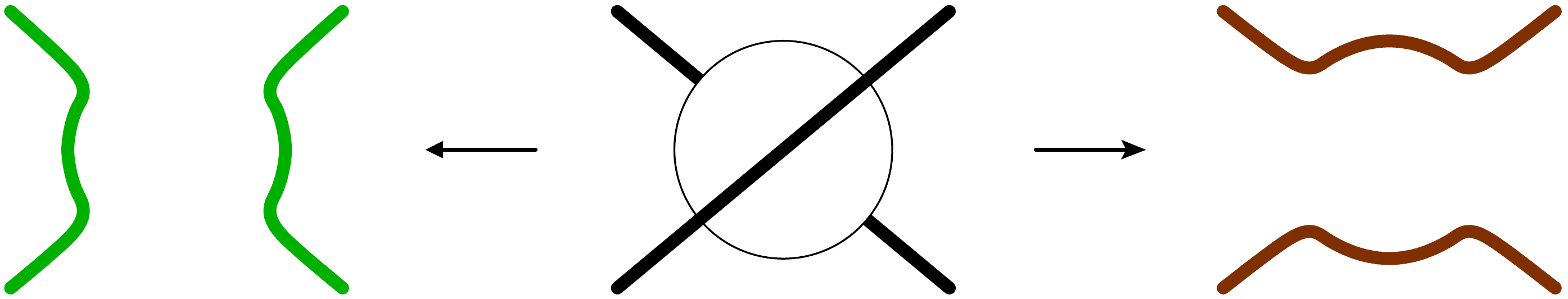}}
\caption{The A-smoothing (left) and B-smoothing (right) of a crossing.}\label{Fi:smooth}
\end{center}
\end{figure}

\subsection{Turaev surfaces}\label{Se:tur}
Each crossing in a link diagram $D$ on a  surface $F$ can be smoothed in two different ways, by inserting a crossing ball $C_i$ and replacing $D\cap C_i$ with one of the two pairs of arcs of $(\partial C_i\cap F)\setminus D$ opposite to another. Figure~\ref{Fi:smooth} shows the two possibilities, called the {\it A-smoothing}  and the {\it B-smoothing} of the crossing.  Making a choice of 
smoothing for each crossing in the diagram produces a disjoint union of circles on 
$F$, called a {\it state} of the diagram $D$. Two states of $D$ are 
{\it dual} if they have opposite smoothings at each crossing.


\begin{figure}
\begin{center}
\scalebox{.21}{\includegraphics{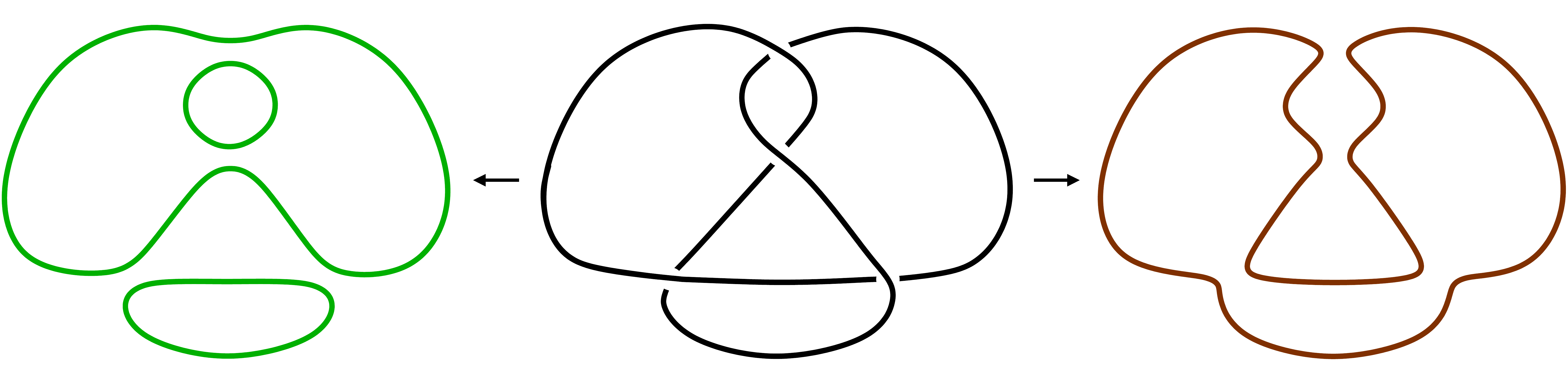}}
\caption{The all-A (left) and all-B (right) states for a link diagram. 
}\label{Fi:ABstates}
\end{center}
\end{figure}

Given a link diagram $D$ on $S^2$, the two extreme states -- the all-A and the all-B -- 
are 
of particular interest,
 due in part to the bounds they give on the maximum and minimum
degrees of the Jones polynomial.  Kauffman's proof~\cite{k} that these bounds are sharp for reduced, alternating diagrams provided the impetus for Murasugi~\cite{mur}, Thistlethwaite~\cite{this}, and Turaev~\cite{tur} to prove Tait's conjecture on the crossing numbers of alternating links. 
Cromwell~\cite{c}, Lickorish and Thistlethwaite~\cite{lt} then extended these results to adequate link diagrams. Figure~\ref{Fi:ABstates} shows the all-A and all-B states for the link diagram from Figure~\ref{Fi:quotTur}.

\begin{figure}
\begin{center}
\scalebox{.36}{\includegraphics{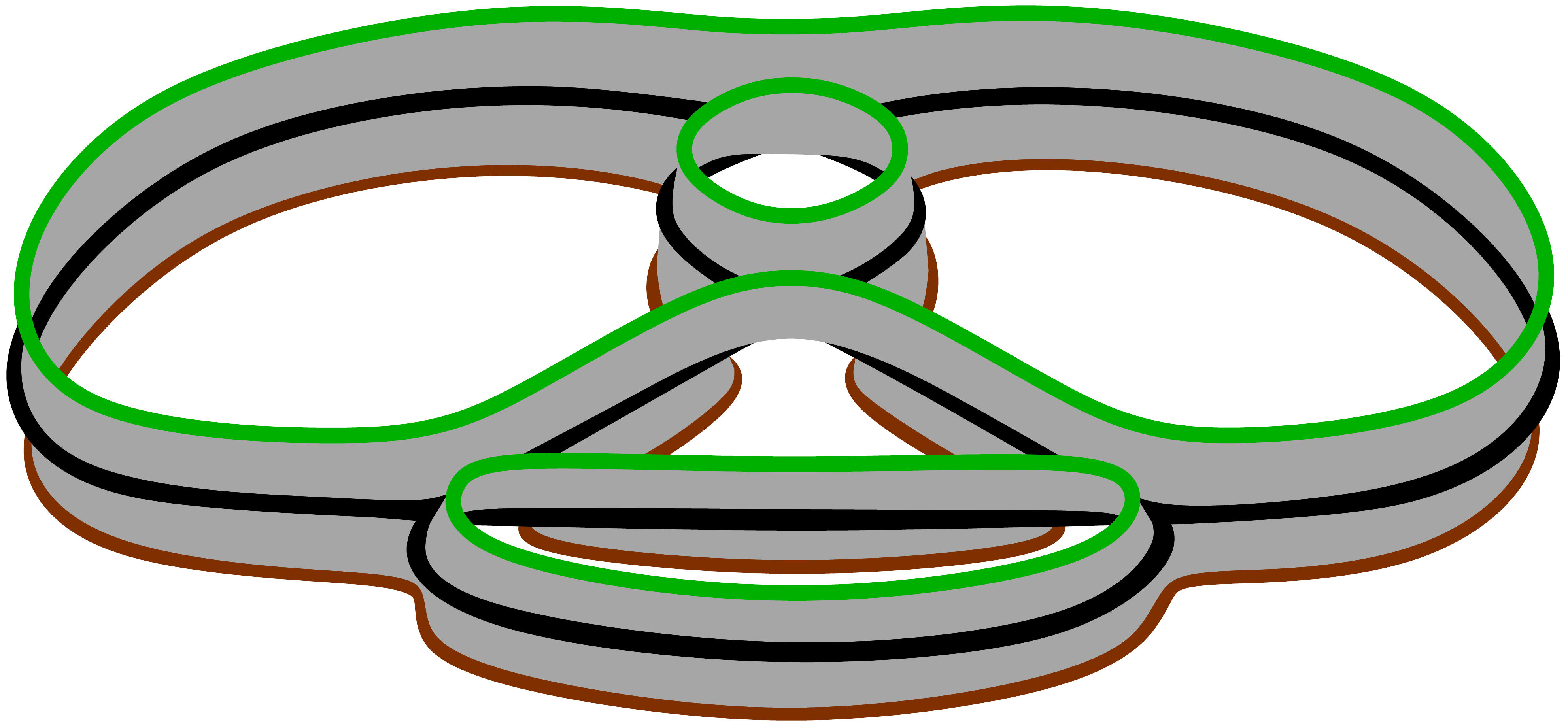} }
\caption{The cobordism between the all-A and all-B states from Figure~\ref{Fi:ABstates}.
}\label{Fi:cobord}
\end{center}
\end{figure}

Following Turaev~\cite{tur}, 
 one can construct a 
 cobordism between the all-A and all-B states 
 as follows.
Parameterize a bi-collaring of $S^2$ as in \textsection\ref{Se:link}, and push the all-A and all-B states off $S^2$ to $S^2\times\{1\}$ and $S^2\times\{-1\}$, respectively, such that each state circle sweeps out an annulus to one side of $S^2$.
Assume that these annuli are mutually disjoint, and that they are disjoint from the crossing balls $C=\bigcup C_i$ used to construct the all-A and all-B states.  
Gluing together these annuli and the disks of $S^2\cap C$ produces the cobordism between the two states. (See Figure~\ref{Fi:cobord}.)  Near each crossing, the cobordism has a saddle, as
in Figure~\ref{Fi:Saddle}.

\begin{figure}
\begin{center}
\scalebox{.195}{\includegraphics{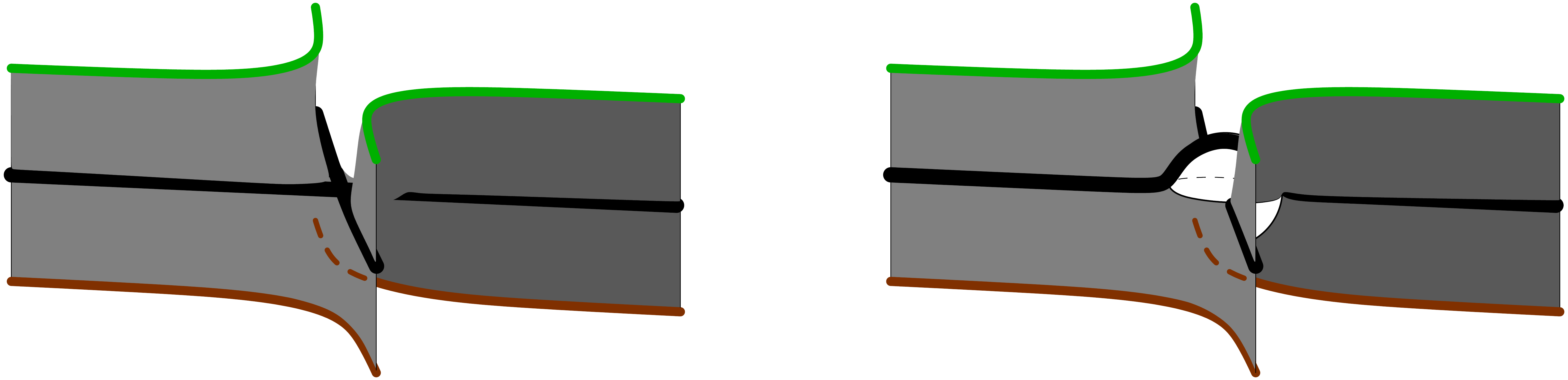} }
\caption{Turaev's cobordism between the all-A and all-B states has a saddle near each crossing, shown here with and without a crossing ball.
}\label{Fi:Saddle}
\end{center}
\end{figure}

Having constructed the cobordism, 
one caps the all-A and all-B states with mutually disjoint 
disks 
to form a closed surface $\Sigma$, 
called the {\it Turaev surface} of the original link diagram $D$ on $S^2$.   Since $\Sigma$ contains a neighborhood of $S^2$ around each crossing point, the crossing information 
of $D$ on $S^2$ translates to crossing information on the Turaev surface. Thus, $D$ forms a link diagram on $\Sigma$. 
A crossing ball configuration corresponding to this link diagram is  $K\subset(\Sigma\setminus C)\cup \partial C$, with under- and over-passes defined as in \textsection\ref{Se:link}.

\begin{figure}
\begin{center}
\scalebox{.58}{\includegraphics{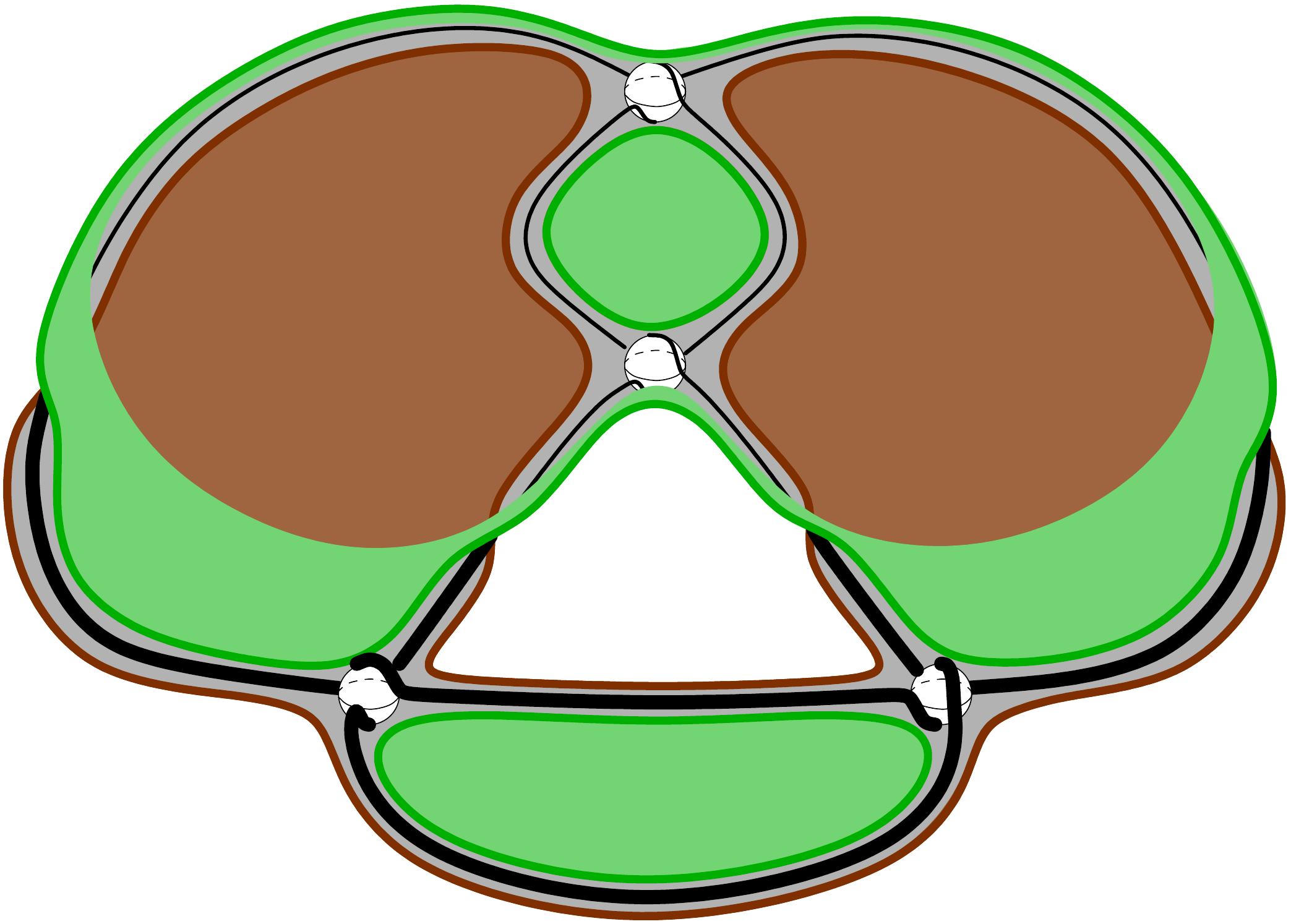} }
\caption{This torus is the Turaev surface of the link diagram in Figures \ref{Fi:quotTur} and \ref{Fi:ABstates}, seen from the ambient space. To provide a window to the far side of the surface, one of the three disks of the all-A state is only partly shown.
}\label{Fi:Tur}
\end{center}
\end{figure}

Observe that $D$ cuts $\Sigma$ into disks, each of which contains exactly one state disk, and that $S^2\cap \Sigma=S^2\cap(C\cup K)=\Sigma\cap (C\cup K)$. Note also that if $D$ is alternating on $S^2$, then $\Sigma$ is a sphere which can be isotoped to $S^2$ while fixing $D$.
Figure~\ref{Fi:Tur} shows a less trivial example.
 
The construction of the Turaev surface generalizes to any pair of states $s$ and $\tilde{s}$ dual to one another. By pushing $s$ and $\tilde{s}$ to opposite sides of $S^2$ to sweep out annuli on opposite sides of $S^2$, gluing in disks near the crossings to obtain a cobordism between $s$ and $\tilde{s}$, and capping off 
with disks, one obtains a closed surface $\Sigma$ on which $D$ forms a link diagram~\cite{bm, tur}.  Call this surface $\Sigma$ the {\it generalized Turaev surface} of the dual states $s$ and $\tilde{s}$.

\section{Construction 
of Heegaard diagrams for Turaev surfaces}\label{Se:3}

Given a connected link diagram $D$ on $S^2\subset S^3$ and its Turaev surface $\Sigma$, this section constructs a 
link-adapted Heegaard diagram $(\Sigma,\alpha,\beta,D)$.  Theorem \ref{Th:5} then characterizes this diagram, providing one direction of the correspondence to come in Theorem \ref{Th:Main1}.  
 
Let $K\subset (S^2\setminus C)\cup \partial C$ be a crossing ball structure corresponding to 
$D$
, and let $H_\alpha$ and $H_\beta$ be the two components of $S^3\setminus \Sigma$.  Define $\hat{\alpha}:=(S^2\setminus (C\cup K))\cap H_\alpha$ and $\hat{\beta}:=(S^2\setminus (C\cup K))\cap H_\beta$ to be the two checkerboard classes of $S^2\setminus (C\cup K)$, with $\alpha:=\partial\hat{\alpha}$ and $\beta:=\partial\hat{\beta}$.  
From this setup, three modifications will complete the construction of the diagram $(\Sigma,\alpha,\beta,D)$. 
During these changes, $\Sigma$, $D$, $S^2$, $C$, and $K$ will remain fixed.


First, perturb $\alpha$ and $\beta$ through the cobordism as follows,  carrying along the disks of $\hat{\alpha}$ and $\hat{\beta}$. Let $X=\{x_1,\hdots, x_n\}$ consist of one point on each arc of $K\setminus C$ which joins two under-passes on $S^2$, and let $Y=\{y_1,\hdots, y_n\}$ consist of one point on each arc of $K\setminus C$ which joins two over-passes on $S^2$.  
Each arc of ${\alpha}\setminus (X\cup Y)$ 
runs 
along a circle from either the all-A state or the all-B state.
Isotope ${\alpha}$ through 
the cobordism so as to push arcs of the former type to $S^2\times (0,1)$ and arcs of the latter type to $S^2\times (-1,0)$, giving $\alpha\cap C=\varnothing$ and ${\alpha}\cap D=X\cup Y$.  Next, isotope ${\beta}$ in the same manner, 
after which $\alpha$ and $\beta$ will both be disjoint from $C$, while $\alpha$, $\beta$, and $D$ will be pairwise transverse and will intersect exclusively at triple points: 
${\alpha}\cap{\beta}={\alpha}\cap D={\beta}\cap D=X\cup Y$.

To further simplify the picture, push the state circles through the cobordism to align with $\alpha \cup \beta$, so that each state disk becomes a component of $\Sigma\setminus (\alpha\cup \beta)$. This causes the neighborhood of each arc of $K\setminus C$ to appear as in Figure~\ref{Fi:altTur}, possibly with red and blue reversed. 
 Note that the state disks' interiors remain disjoint from $D$, in fact from $S^2$.

To complete the construction, remove 
 any attaching circles that are disjoint from 
 $D$.   
Also remove the corresponding disks of $\hat{\alpha}$ and $\hat{\beta}$, and let $\alpha$, $\beta$, $\hat{\alpha}$ and $\hat{\beta}$ retain their names. 
Because each removed circle lies in some disk of $\Sigma\setminus D$,  each removed disk is parallel to $\Sigma$. 

\begin{lemma}[DFKLS~\cite{dfkls}]\label{L:1}
The Turaev surface $\Sigma$ of any connected link diagram $D$ on $S^2\subset S^3$ is a splitting surface for $S^3$.
\end{lemma}

\begin{proof}
Observe that $S^2\cup C$ cuts $S^3$ into two balls, which $\Sigma$ cuts into smaller balls.  Also, $S^3\setminus (S^2\cup C\cup \Sigma)=(H_\alpha\setminus (S^2\cup C))\cup (H_\beta\setminus (S^2\cup C))$, where $H_\alpha$ and $H_\beta$ are the two components of $S^3\setminus \Sigma$.  
Hence, $H_\alpha\setminus C$ and $H_\beta\setminus C$ are handlebodies, as are $H_\alpha$ and $H_\beta$. 
\end{proof}

The proof of Lemma \ref{L:1} implies that $(\Sigma,\alpha,\beta)$ was a Heegaard diagram for $S^3$ when $\alpha$ and $\beta$ were first defined.  The fact that each removed disk of $\hat{\alpha}$ and of $\hat{\beta}$ was parallel to $\Sigma$ implies that $(\Sigma,\alpha,\beta)$ is a Heegaard diagram for $S^3$ in the finished construction as well.

\newpage

\begin{figure}
\begin{center}
\scalebox{.125}{\includegraphics{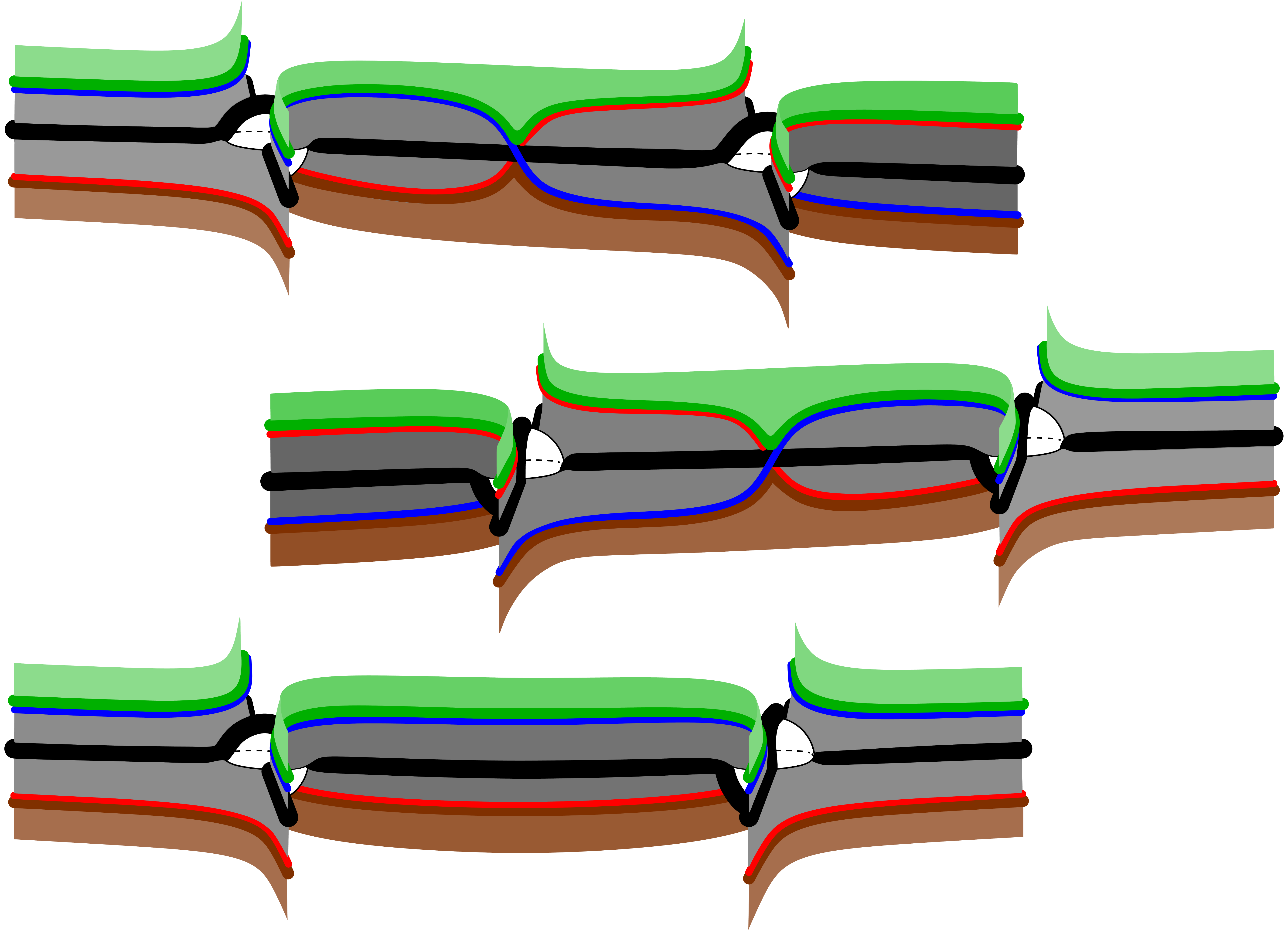} }
\caption{Up to reversing red and blue, these are the possible configurations of the Turaev surface $\Sigma$ between two adjacent crossings, shown at the stage of the construction in which the boundary of each state disk lies in $\alpha\cup\beta$.
}\label{Fi:altTur}
\end{center}
\end{figure}

\begin{lemma}[DFKLS~\cite{dfkls}]\label{L:2}
Any connected link diagram $D$ on $S^2\subset S^3$ forms an alternating link diagram on its Turaev surface $\Sigma$.
\end{lemma}

\begin{proof}
Recall from \textsection\ref{Se:tur} that $D$ forms a link diagram on $\Sigma$.  On $S^2$, each arc $\kappa$ of $K\setminus C$ joins either two over-passes, two under-passes, or one of each.  Figure~\ref{Fi:altTur} shows the three possible configurations of $\Sigma$ near $\kappa$, prior to the removal of attaching circles, up to reversal of $\alpha$ and $\beta$. In all three cases, the two arcs of $K\cap \partial C$ incident to $\kappa$ lie to opposite sides of $\Sigma$, so that one is an over-pass on $\Sigma$ and the other is an under-pass on $\Sigma$.
\end{proof}

One defines the {\it Turaev genus} $g_T(K)$ of a link $K\subset S^3$ to be the minimum genus among the Turaev surfaces of all diagrams of $K$ on $S^2$. 
The resulting invariant, surveyed in~\cite{ck14}, measures how far a link is from being alternating. See also~\cite{balm}. In particular, Turaev genus provides the crux of Turaev's proof of Tait's conjecture:

\begin{corollary}[Turaev~\cite{tur}, DFKLS~\cite{dfkls}]
A link $K$ is alternating if and only if $g_T(K)=0$.
\end{corollary}

\newpage

\begin{theorem}\label{Th:5}
From the Turaev surface $\Sigma$ of a connected link diagram $D$ on $S^2\subset S^3$, the construction in this section produces a diagram $(\Sigma,\alpha,\beta, D)$ with the following properties:

$\bullet$ $(\Sigma,\alpha,\beta)$ is a Heegaard diagram for $S^3$, with $\alpha\pitchfork \beta$.

$\bullet$ $D$ is an alternating link diagram on $\Sigma$ which cuts $\Sigma$ into disks, with $D\pitchfork \alpha$ and $D\pitchfork \beta$. 

$\bullet$  $D\cap\alpha=D\cap\beta=\alpha\cap \beta$, none of these points being crossings of $D$. 

$\bullet$ There is a checkerboard partition $\Sigma\setminus(\alpha\cup\beta)=\Sigma_\varnothing\cup \Sigma_K$, in which 
$\Sigma_\varnothing$ consists of disks disjoint from $D$, in which $D$ cuts $\Sigma_K$ into disks each of whose boundary contains at least one crossing point and at most two points of $\alpha\cap \beta$, and in which $ 2  g(\Sigma)+|\Sigma_\varnothing|=|\alpha|+|\beta|$.
 \end{theorem}
\begin{proof}
We have already established the first three properties.  Let $\Sigma_\varnothing$ consist of the interiors of all adjusted 
state disks whose boundary contains at least one point of $\alpha\cap \beta$, i.e. those whose boundary still lies in $\alpha\cup\beta$ after the removal of the attaching circles disjoint from $D$. 
These state disks are disjoint from $D$ and 
constitute a checkerboard class of $\Sigma\setminus(\alpha\cup \beta)$. 
See Figure~\ref{Fi:TurSurg}. 

Let $\Sigma_K$ denote the other checkerboard class of $\Sigma\setminus (\alpha\cup\beta)$.
Each component of $\Sigma_K\setminus D$ is also a component of $(\Sigma\setminus D)\setminus(\alpha\cup\beta)$, and each attaching circle intersects $D$; therefore, $D$ cuts $\Sigma_K$ into disks. Further, each arc of $K\setminus C$ contains at most one point of $\alpha\cap \beta$, and each arc of $(\alpha\cup\beta)\setminus D$
is parallel through $\Sigma$ to $D$; consequently, the boundary of each disk of $\Sigma_K\setminus D$ contains at least one crossing point and at most one arc of $(\alpha\cup\beta)\setminus D$, hence at most two points of $\alpha\cap\beta$.
 
 Finally, to see that $2g(\Sigma)+|\Sigma_\varnothing|=|\alpha|+|\beta|$, consider Euler characteristic in light of the observation that removing the disks of
  $\Sigma_\varnothing$ from $\Sigma$ and gluing in the disks of 
  $\hat{\alpha}$ and $\hat{\beta}$ 
yields a sphere isotopic to $S^2$. Near each point of $\alpha\cap \beta$, this surgery appears as in Figure~\ref{Fi:TurSurg}. 
\end{proof}

\begin{figure}
\begin{center}
\scalebox{.3}{\includegraphics{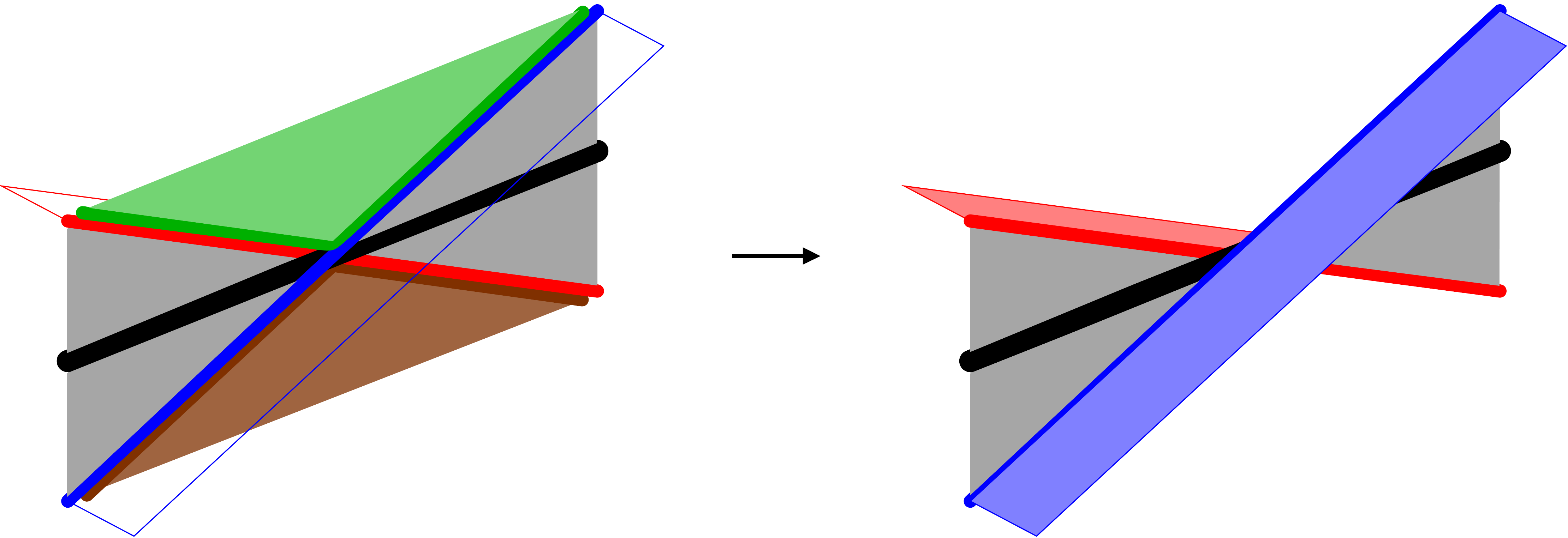}  }
\caption{Given a diagram $(\Sigma,\alpha,\beta,D)$ with the properties in Theorems \ref{Th:5}, \ref{Th:Main1}, or  \ref{Th:Gen}, removing the disks of $\Sigma_\varnothing$ from $\Sigma$ and gluing in the disks of $\hat{\alpha}$ and $\hat{\beta}$ produces a sphere on which $D$ forms a link diagram. 
Near each point of $\alpha\cap \beta$, this surgery appears as shown, up to mirroring.
}\label{Fi:TurSurg}
\end{center}
\end{figure}

\section{Correspondence between Heegaard diagrams and Turaev surfaces}

\subsection{Main correspondence}
From the Turaev surface 
of a connected
link diagram on $S^2\subset S^3$, we have constructed a link-adapted Heegaard diagram $(\Sigma,\alpha,\beta, D)$ with several properties.  We will now see that any such diagram corresponds to the Turaev surface of some link diagram on the sphere.

\begin{theorem}\label{Th:Main1}
There is a 1-to-1 correspondence between Turaev surfaces of connected link diagrams on $S^2\subset S^3$ and diagrams $(\Sigma,\alpha,\beta,D)$ with the following properties:

$\bullet$ $(\Sigma,\alpha,\beta)$ is a Heegaard diagram for $S^3$, with $\alpha\pitchfork \beta$.

$\bullet$ $D$ is an alternating link diagram on $\Sigma$ which cuts $\Sigma$ into disks, with $D\pitchfork \alpha$ and $D\pitchfork \beta$. 

$\bullet$  $D\cap\alpha=D\cap\beta=\alpha\cap \beta$, none of these points being crossings of $D$. 

$\bullet$ There is a checkerboard partition $\Sigma\setminus(\alpha\cup\beta)=\Sigma_\varnothing\cup \Sigma_K$, in which 
$\Sigma_\varnothing$ consists of disks disjoint from $D$, in which $D$ cuts $\Sigma_K$ into disks each of whose boundary contains at least one crossing point and at most two points of $\alpha\cap \beta$, and in which $ 2  g(\Sigma)+|\Sigma_\varnothing|=|\alpha|+|\beta|$.
\end{theorem}

\begin{proof}
Theorem \ref{Th:5} provides one direction of this correspondence.  It remains to prove the converse.  

Assume that the diagram $(\Sigma,\alpha,\beta,D)$ is as described.  Remove the disks of $\Sigma_\varnothing$ from $\Sigma$ and glue in the disks of $\hat{\alpha}$ and $\hat{\beta}$ to obtain a closed surface. (See Figure~\ref{Fi:TurSurg}.)   Because $D$ is connected and $ 2  g(\Sigma)+|\Sigma_\varnothing|=|\alpha|+|\beta|$, this surface is a sphere -- call it $S^2$.  Moreover, $D$, being disjoint from $\Sigma_\varnothing$ and having its crossing points in $\Sigma_K$, forms a link diagram on $S^2$.  We claim, up to isotopy, that $\Sigma$ is the Turaev surface of the link diagram $D$ on $S^2$.

The property that $D$ cuts $\Sigma_K$ into disks implies that $D$ intersects each attaching circle, cutting $\alpha$ and $\beta$ into arcs.  Because the boundary of each disk of $\Sigma_K\setminus D$ contains at most two points of $\alpha\cap \beta$, each of these arcs is parallel through one of these disks to $D$.  The property that the boundary of each disk of $\Sigma_K\setminus D$ contains at least one crossing point then implies that there is at most one point of $\alpha\cap\beta$ on $D$ between any two adjacent crossings. 
 
The link diagram $D$ cuts $S^2$ into disks admitting a checkerboard partition.  Because $S^2$ appears near each point of $\alpha\cap \beta$ as in Figure~\ref{Fi:TurSurg}, one of the checkerboard classes contains
 $\hat{\alpha}$, and the other contains 
 $\hat{\beta}$.  
Yet, some disks of $S^2\setminus D$ may be entirely contained in $\Sigma_K$, hence disjoint from $\alpha$ and $\beta$.  
Construct an attaching circle in the interior of each such disk, and incorporate it into either $\alpha$ or $\beta$ according to the checkerboard pattern, letting $\alpha$ and $\beta$ retain their names.   Span each new circle of $\alpha$ by a new disk of $\hat{\alpha}$ on the same side of $\Sigma$ as the other disks of $\hat{\alpha}$, and similarly span each new circle of $\beta$ by a new disk of $\hat{\beta}$. 

The components of  $\Sigma\setminus (\alpha\cup\beta)$ still admit a checkerboard partition, $\Sigma\setminus (\alpha\cup \beta)=\Sigma_\varnothing\cup \Sigma_K$, in which $\Sigma_\varnothing$ consists of disks disjoint from $D$, though $D$ no longer need cut $\Sigma_K$ into disks. 
The preceding modification of $\alpha$, $\beta$, $\hat{\alpha}$, and $\hat{\beta}$ corresponds to an isotopy of $S^2$, which again may be obtained from $\Sigma$ by removing the disks of $\Sigma_\varnothing$ and gluing in the disks of $\hat{\alpha}$ and $\hat{\beta}$.

Let $K\subset (\Sigma\setminus C)\cup\partial C$ be a crossing ball configuration corresponding to the link diagram $D$ on $\Sigma$, with $C\cap\alpha=\varnothing=C\cap \beta$.  Note that $K\subset (S^2\setminus C)\cup\partial C$ is also a crossing ball configuration corresponding to the link diagram $D$ on $S^2$.

Currently $\Sigma$ and $S^2$ are non-transverse, even away from $C$, as both $\Sigma$ and $S^2$ contain $\Sigma_K$. Rectify this by perturbing $S^2$ as follows, fixing $\Sigma$, $\alpha$, $\beta$, $\hat{\alpha}$, $\hat{\beta}$, $D$, $\Sigma_\varnothing$, $\Sigma_K$, $C$, and $K$ in the process.  
(We initially constructed $S^2$ by gluing together $\hat{\alpha}$, $\hat{\beta}$, and $\Sigma_K$, but now we are pushing $S^2$ off of them.)
Each disk of $S^2\setminus(C\cup K)$ currently contains a disk of either $\hat{\alpha}$ or $\hat{\beta}$; push the disk of $S^2\setminus (C\cup K)$ off $\Sigma$ in the corresponding direction, while fixing its boundary, which lies in $\Sigma\cap(K\cup \partial C)$.  This isotopy makes $S^2$ disjoint from $\alpha$ and $\beta$, except at the points of $\alpha\cap \beta$.  In fact, this isotopy gives $S^2\cap \Sigma=S^2\cap(C\cup K)=\Sigma\cap(C\cup K)$, as was the case 
in \textsection\ref{Se:tur}.  (Recall Figure~\ref{Fi:Saddle}.)

Because $D$ is alternating on $\Sigma$, the disks of $\Sigma\setminus (C\cup K)$ admit a checkerboard partition -- the boundaries of the disks in the two classes are the all-A and all-B state circles for the link diagram $D$ on $\Sigma$.
Further, each of these state circles on $\Sigma$ encloses precisely one disk of $\Sigma_\varnothing$.  Color green each disk of $\Sigma_\varnothing$ enclosed by a circle from the all-A state, and color brown each disk of $\Sigma_\varnothing$ enclosed by a circle from the all-B state. Near each arc of $K\setminus C$, $\Sigma$ now appears as in Figure~\ref{Fi:altTur} (possibly with red and blue reversed). As 
 a final adjustment, slightly perturb the green and brown disks so that they become disjoint from $\alpha$, $\beta$, and $D$. 

 Removing the green and brown disks from $\Sigma$ leaves a cobordism between
 their boundaries. 
 Cutting this cobordism along $S^2\cap(K\cup \partial C)$ yields the disks of $S^2\setminus C$, together with annuli lying to either side of $S^2$, through which the boundaries of the green and brown disks are respectively parallel to the all-A and all-B states of the  link diagram $D$ on $S^2$.
 As claimed, $\Sigma$ is therefore the Turaev surface of the link diagram $D$ on $S^2$.
\end{proof}

\subsection{Generalization to arbitrary dual states}

As noted at the end of \textsection\ref{Se:tur}, the construction of the Turaev surface from the all-A and all-B states of a link diagram $D$ on $S^2$ generalizes to 
any pair of states of $D$ which are  dual to one another, having opposite smoothings at each crossing.  
The correspondence developed  in \textsection3 and \textsection4.1 between link-adapted Heegaard diagrams $(\Sigma,\alpha,\beta,D)$ and Turaev surfaces extends to these generalized Turaev surfaces, the only difference being that $D$ no longer need alternate on $\Sigma$.

\begin{theorem}\label{Th:Gen}
There is a 1-to-1 correspondence between generalized Turaev surfaces 
of connected link diagrams on $S^2\subset S^3$, and diagrams $(\Sigma,\alpha,\beta,D)$ with the following properties:

$\bullet$ $(\Sigma,\alpha,\beta)$ is a Heegaard diagram for $S^3$, with $\alpha\pitchfork \beta$.

$\bullet$ $D$ is a link diagram on $\Sigma$ which cuts $\Sigma$ into disks, with $D\pitchfork \alpha$ and $D\pitchfork \beta$. 

$\bullet$  $D\cap\alpha=D\cap\beta=\alpha\cap \beta$, none of these points being crossings of $D$. 

$\bullet$ There is a checkerboard partition $\Sigma\setminus(\alpha\cup\beta)=\Sigma_\varnothing\cup \Sigma_K$, in which 
$\Sigma_\varnothing$ consists of disks disjoint from $D$, in which $D$ cuts $\Sigma_K$ into disks each of whose boundary contains at least one crossing point and at most two points of $\alpha\cap \beta$, and in which $ 2  g(\Sigma)+|\Sigma_\varnothing|=|\alpha|+|\beta|$.
\end{theorem}

\begin{proof}
Given the generalized Turaev surface $\Sigma$ for dual states $s$ and $\tilde{s}$ of a connected link diagram $D$ on $S^2\subset S^3$, reverse some collection of the crossings of $D$ 
to obtain a new link diagram $D'$ for which $s$ and $\tilde{s}$ are the all-A and all-B states.   Construct the corresponding diagram $(\Sigma,\alpha,\beta,D')$ as in 
\textsection\ref{Se:3}.  Finally, switch back the reversed crossings of $D'$ to obtain the required diagram $(\Sigma,\alpha,\beta,D)$.

Conversely, suppose that $(\Sigma,\alpha,\beta,D)$ is as described.  The proof of Theorem \ref{Th:Main1} extends 
almost verbatim.  The only concern, as $D$ need not alternate on $\Sigma$, is whether or not 
the disks of $\Sigma\setminus D$ admit a checkerboard partition; it suffices to show that they do.

The condition that $D\cap \Sigma_\varnothing=\varnothing$ implies that one endpoint of each arc of $(\alpha\cup\beta)\setminus D$  appears as in Figure~\ref{Fi:TurSurg}, and the other appears as the mirror image.  Thus,   each attaching circle intersects $D$ in an even number of points.  The fact that the attaching circles generate $H_1(\Sigma)$ then implies that any simple closed curve on $\Sigma$ in general position with respect to $D$  must intersect $D$ in an even number of points, and hence that the disks of $\Sigma\setminus D$ admit a checkerboard partition. \end{proof}

\subsection{Conclusion} 
Up to isotopy, each link diagram on $S^2\subset S^3$ has a unique Turaev surface. Theorem \ref{Th:Main1} thus establishes -- via Turaev surfaces -- a 1-to-1 correspondence between link diagrams on $S^2\subset S^3$ and alternating, link-adapted Heegaard diagrams $(\Sigma,\alpha,\beta,D)$.

Similarly, Theorem \ref{Th:Gen} establishes -- via generalized Turaev surfaces constructed from dual states -- a 2-to-1 correspondence between states of link diagrams on $S^2\subset S^3$ and link-adapted Heegaard diagrams $(\Sigma,\alpha,\beta,D)$ for $S^3$ in which $D$ need not alternate on $\Sigma$.

\section{Appendix}
Let  $(\Sigma,\alpha,\beta)$ be a Heegaard diagram for $S^3$, 
and let $\gamma\subset \Sigma$ be an oriented, simple closed curve.  The following construction yields an expression for $[\gamma]\in H_1 (\Sigma)$ in terms of the homology classes of the oriented attaching circles, proving that the latter generate $H_1(\Sigma)$.

Because $H_1(S^3)$ is trivial, $\gamma$ bounds a Seifert surface $S\subset S^3$, on which $\gamma$ induces an orientation. Fixing $\gamma$, isotope $S$ so that its interior intersects $\Sigma$ transversally -- along simple closed curves and along arcs with endpoints on $\gamma$.  

Given a component $S_{\alpha,i}$ of $S\cap H_{\alpha}$, one may obtain an expression for $[\partial S_{\alpha,i}]\in H_1(\Sigma)$ in terms of the $[\alpha_j]$ by surgering $S_{\alpha,i}$ along successive outermost disks of $\hat{\alpha}\setminus S_{\alpha,i}$ until $\partial S_{\alpha,i}$ lies entirely in the punctured sphere $\Sigma\setminus \alpha$, at which point the expression is evident. An analogous procedure expresses the homology class of each component of $S\cap H_\beta$ in terms of the $[\beta_j]$.   
Summing over all components of $S\setminus \Sigma$ gives the desired expression for $[\gamma]\in H_1(\Sigma)$:
\[[\gamma]=[\partial S]=\sum_{\begin{matrix} \small{\text{Components }} \\ \small{S_{\alpha,i}\text{ of }S\cap H_\alpha} \\ \end{matrix}}[\partial S_{\alpha,i}]+\sum_{\begin{matrix} \small{\text{Components }} \\ \small{S_{\beta,i}\text{ of }S\cap H_\beta} \\ \end{matrix}}[\partial S_{\beta,i}]=\sum_{i,j} a_{i,j}[\alpha_j]+\sum_{i,j} b_{i,j}[\beta_j]\]
Conversely, if $(\Sigma,\alpha,\beta)$ is a Heegaard diagram for a 3-manifold $M$ with nontrivial first homology, then the oriented attaching circles do not  generate $H_1(\Sigma)$, since inclusion $\Sigma\hookrightarrow M$ induces a surjective map $H_1(\Sigma)\to H_1(M)$, whose kernel contains all the $[\alpha_j]$ and $[\beta_j]$.





\begin{thebibliography}{99}

\bibitem{bm} Y. Bae, H.R. Morton, {\it The spread and extreme terms of Jones polynomials}, J. Knot Theory Ramifications 12 (2003), 359-373.

\bibitem{balm} C.L.J. Balm, {\it Topics in knot theory: On generalized crossing changes and the additivity of the Turaev genus}, Thesis (Ph.D.) -- Michigan State University (2013). 

\bibitem{ck} A. Champanerkar, I. Kofman, {\it Spanning trees and Khovanov homology}, Proc. Amer. Math. Soc. 137 (2009), no. 6, 2157-2167.

\bibitem{ck14} A. Champanerkar, I. Kofman, {\it A survey on the Turaev genus of knots}, arXiv:1406.1945, preprint.

\bibitem{cks07} A. Champanerkar, I. Kofman, N. Stoltzfus, {\it Graphs on surfaces and Khovanov homology}, Algebr. and Geom. Topol. 7 (2007), 1531-1540.

\bibitem{cks11} A. Champanerkar, I. Kofman, N. Stoltzfus, {\it Quasi-tree expansion for the Bollob$\acute{\text{a}}$s-Riordan-Tutte polynomial}, Bull. Lond. Math. Soc. 43 (2011), no. 5, 972-984.

\bibitem{c} P.R. Cromwell, {\it Homogeneous links}, J. London Math. Soc. (2) 39 (1989), no. 3, 535-552.

\bibitem{dfkls} O.T. Dasbach, D. Futer, E. Kalfagianni, X.-S. Lin,  N. Stoltzfus, {\it The Jones polynomial and graphs on surfaces}, J. Combin. Theory Ser. B 98 (2008), no. 2, 384-399. 

\bibitem{dl11} O.T. Dasbach, A. Lowrance, {\it Turaev genus, knot signature, and the knot homology concordance invariants}, Proc. Amer. Math. Soc. 139 (2011), no. 7, 2631-2645.

\bibitem{dl13} O.T. Dasbach, A. Lowrance, {\it A Turaev surface approach to Khovanov homology}, arXiv:1107.2344v2.


\bibitem{j} V.F.R. Jones, {\it A polynomial invariant for knots via Von Neumann algebras}, Bull. Amer. math. Soc. (N.S.) 12 (1985), no. 1, 103-111.

\bibitem{k} L.H. Kauffman, {\it State models and the Jones polynomial}, Topology 26 (1987), no. 3, 395-407.

\bibitem{lt} W.B.R. Lickorish and M.B. Thistlethwaite, {\it Some links with non-trivial polynomials and
their crossing numbers}, Comment. Math. Helv. 63 (1988), no. 4, 527-539.

\bibitem{low} A. Lowrance, {\it On knot Floer width and Turaev genus}, Algebr. Geom. Topol. 8 (2008), no. 2, 1141-1162.

\bibitem{men} W. Menasco, {\it Closed incompressible surfaces in alternating knot and link complements}, Topology 23 (1984), no. 1, 37-44.

\bibitem{mur} K. Murasugi, {\it Jones polynomials and classical conjectures in knot theory}, Topology 26 (1987), no. 2, 187-194.



\bibitem{tait} P.G. Tait, {\it On Knots I, II, and III}, Scientific papers 1 (1898), 273-347.

\bibitem{this} M.B. Thistlethwaite, {\it A spanning tree expansion of the Jones polynomial}, Topology 26 (1987), no. 3, 297-309.

 \bibitem{tur} V.G. Turaev, {\it A simple proof of the Murasugi and Kauffman theorems on alternating links}, Enseign. Math. (2) 33 (1987), no. 3-4, 203-225.

\bibitem{weh} S. Wehrli, {\it A spanning tree model for Khovanov homology}, J. Knot Theory Ramifications 17 (2008), no. 12, 1561-1574.








\end{thebibliography}
\end{document}